\documentclass[12pt]{amsart}
\usepackage[utf8]{inputenc}
\usepackage{amsmath,amssymb,amsthm}
\usepackage{mathrsfs}
\usepackage{geometry}
\usepackage{graphicx}  
\usepackage{hyperref}
\usepackage{cite}
\usepackage{pgfplots}
\pgfplotsset{compat=1.18}
\newtheorem{theorem}{Theorem}[section]
\newtheorem{lemma}[theorem]{Lemma}
\newtheorem{proposition}[theorem]{Proposition}

\theoremstyle{remark}
\newtheorem{remark}{Remark}
\theoremstyle{definition}
\newtheorem{definition}[theorem]{Definition}
\newtheorem{assumption}[theorem]{Assumption}
\newtheorem{example}[theorem]{Example}

\numberwithin{equation}{section}

\title{Stochastic Stability of ACIMs for Piecewise Expanding $C^{1+\varepsilon}$ Maps}
\author{Aparna Rajput}
\address{Department of Mathematics and Statistics, Concordia University, Montreal, H3G 1M8, QC, Canada}
\email{aparna.ar.rajput@gmail.com}
\date{\today}
\begin{document}

\begin{abstract}
We prove stochastic stability of absolutely continuous invariant measures (ACIMs) for piecewise expanding $C^{1+\varepsilon}$ maps of the interval. For maps $\tau$ in the class $\mathcal{T}([0,1]; s, \varepsilon)$, we consider perturbed Frobenius--Perron operators $P_\delta = Q_\delta P_\tau$, where $Q_\delta$ is a Markov smoothing operator modeling noise of intensity $\delta > 0$.

In the generalized bounded variation space $BV_{1,1/p}(I)$, we establish a Lasota--Yorke inequality uniform in $\delta$. Consequently, each $P_\delta$ admits an invariant density $h_\delta \in BV_{1,1/p}(I)$, and $h_\delta \to h$ in $L^1(I)$ as $\delta \to 0$, where $h$ is the ACIM density of $P_\tau$.

 Our proof combines the $BV_{1,1/p}(I)$ framework with uniform quasi-compactness and the
  Keller--Liverani perturbation theory for transfer operators. Under a topological mixing
  assumption on~$\tau$, this establishes stochastic stability at the $C^{1+\varepsilon}$ threshold,
  which is essentially optimal for ACIM existence: $C^1$ regularity alone is insufficient to
  guarantee the existence of an ACIM.
\end{abstract}
\maketitle
\section{Introduction}

Piecewise expanding maps of the interval form a fundamental class of dynamical systems in one dimension. A central object in their study is an absolutely continuous invariant measure (ACIM), whose density describes the long-term statistical behavior of almost all initial conditions. A classical result of Lasota and Yorke~\cite{lasota-yorke} establishes the existence of ACIMs for sufficiently smooth maps with uniform expansion via what is now known as the Lasota--Yorke inequality. This approach provides control over both variation and the \(L^1(I)\) norm of densities and leads to strong statistical properties such as mixing and decay of correlations. Subsequent works by Pianigiani-Yorke~\cite{pianigiani1979}, Hofbauer--Keller~\cite{hofbauer1982}, and Keller~\cite{keller1985} extended these results to broader settings, typically under higher regularity assumptions.

More recently, attention has turned to systems with minimal smoothness. In particular, Rajput and Góra~\cite{RajputGora2025} proved the existence of ACIMs for piecewise \(C^{1+\varepsilon}\) expanding maps by working in the generalized bounded variation space \(BV_{1,1/p}(I)\) defined via oscillation seminorms:
\begin{equation}\label{eq:BVnorm}
\|f\|_{1,1/p}
=
\sup_{0<r\le 1}\frac{\operatorname{Osc}_1(f,r)}{r^{1/p}}
+
\|f\|_{1}.
\end{equation}
This framework allows one to treat maps whose derivatives are only Hölder continuous, and the assumption of \(C^{1+\varepsilon}\) regularity (\(\varepsilon>0\)) is essentially sharp, as \(C^1\) regularity alone is insufficient to guarantee the existence of ACIMs.

A natural question is whether these invariant measures are stable under small perturbations. This problem, known as stochastic stability, concerns the persistence of invariant densities under the introduction of noise or small perturbations of the dynamics. It plays an important role in understanding the robustness of statistical properties of dynamical systems, particularly in applications where noise or numerical approximation is unavoidable.  Classical results on stochastic stability have been established for uniformly expanding maps
  in spaces such as bounded variation or H\"older spaces, typically under stronger smoothness
  assumptions. In the piecewise $C^{1+\varepsilon}$ setting specifically, G\'ora~\cite{gora1984}
  proved stochastic stability using Rychlik's $C_\varepsilon$ seminorm, under expansion conditions
  more restrictive than those required here. The present paper reproves and strengthens this
  result in the $BV_{1,1/p}(I)$ framework via the Keller--Liverani perturbation theory, with
  weaker hypotheses and a more transparent spectral argument.

The purpose of this paper is to study stochastic stability of ACIMs for piecewise expanding \(C^{1+\varepsilon}\) maps in the above low-regularity setting. Let \(\tau\) be a map in the class \(\mathcal{T}([0,1];s,\varepsilon)\) of piecewise expanding \(C^{1+\varepsilon}\) maps of the interval. The noise operator \(Q_\delta\) is defined via a periodized (torus) convolution kernel, which automatically satisfies the Markov property
\[
\int_I q_\delta^{\mathrm{per}}(x,y)\,dy = 1
\quad \text{for all } x \in I,
\]
without requiring a separate boundary correction. This is the mechanism that allows us to work directly on the interval \(I=[0,1]\). We consider a family of perturbed Frobenius--Perron operators of the form
\begin{equation}\label{eq:perturbedFP}
P_\delta = Q_\delta P_\tau,
\end{equation}
where \(P_\tau\) is the Frobenius--Perron operator associated with \(\tau\), and \(Q_\delta\) is a Markov smoothing operator representing noise of intensity \(\delta>0\). Working in the space \(BV_{1,1/p}(I)\), we establish a Lasota--Yorke inequality that is uniform in \(\delta\). As a consequence, each operator \(P_\delta\) admits an invariant density \(h_\delta \in BV_{1,1/p}(I)\), and we prove that \(\|h_\delta-h\|_{L^1(I)}\to 0\) as \(\delta\to 0\), where \(h\) is the invariant density of the unperturbed system \(P_\tau\) in \(BV_{1,1/p}(I)\).

Our approach combines the generalized bounded variation framework with perturbation techniques for quasi-compact operators. This extends stochastic stability results to piecewise \(C^{1+\varepsilon}\) maps under minimal regularity assumptions in the \(BV_{1,1/p}(I)\) setting.

The paper is organized as follows. In Section 2, we recall the definition and properties of the space \(BV_{1,1/p}(I)\) and define the class \(\mathcal{T}([0,1];s,\varepsilon)\) of piecewise expanding \(C^{1+\varepsilon}\) maps. In Section 3, we define the perturbation operators and establish a uniform Lasota--Yorke inequality. Section 4 proves existence of invariant densities and stochastic stability via convergence
  in $L^1(I)$. Section 5 presents a nonlinear example illustrating the theory.
  Section 6 discusses remarks and further directions.
\section{Preliminaries}

In this section, we recall the definition and basic properties of the generalized bounded variation spaces \(BV_{1,1/p}(I)\), following~\cite{RajputGora2025, keller1985}. These spaces provide the appropriate functional setting for studying the Frobenius--Perron operator associated with piecewise expanding maps of low regularity.

Let \(I = [0,1]\), and let \(m\) denote the Lebesgue measure on \(I\). For a measurable function \(f : I \to \mathbb{R}\) and \(r > 0\), define the oscillation of \(f\) at scale \(r\) by
\begin{equation}\label{eq:osc}
   \operatorname{Osc}(f,r,x)
   =
   \sup \bigl\{ |f(y_1) - f(y_2)| : y_1, y_2 \in (x-r,x+r) \cap I \bigr\}.
\end{equation}
The function \(\operatorname{Osc}(f,r,\cdot)\) is measurable and describes the local variation of \(f\).

For \(1 \le p < \infty\), define
\begin{equation}\label{eq:osc1}
    \operatorname{Osc}_1(f,r)
    =
    \int_I \operatorname{Osc}(f,r,x)\,dm(x).
\end{equation}

\begin{definition}\label{def:BV}
Let \(p \ge 1\). The space \(BV_{1,1/p}(I)\) consists of all functions \(f \in L^1(I)\) such that
\begin{equation}\label{eq:var}
    \operatorname{var}_{1,1/p}(f)
    :=
    \sup_{0<r\le 1}
    \frac{\operatorname{Osc}_1(f,r)}{r^{1/p}}
    < \infty.
\end{equation}
We equip this space with the norm
\begin{equation}\label{eq:BVnorm2}
  \|f\|_{1,1/p}
  =
  \operatorname{var}_{1,1/p}(f) + \|f\|_{1}.
\end{equation}
\end{definition}

The space $\bigl(BV_{1,1/p}(I), \|\cdot\|_{1,1/p}\bigr)$ is a Banach space and is continuously embedded in \(L^1(I)\). Moreover, the unit ball in \(BV_{1,1/p}(I)\) is relatively compact in \(L^1(I)\).

We will use the following basic properties of oscillation.

\begin{proposition}\label{prop:osc-prop}
Let \(f \in BV_{1,1/p}(I)\) and \(r > 0\). Then:
\begin{enumerate}
\item \(\operatorname{Osc}(f,r,\cdot)\) is lower semicontinuous and hence measurable;
\item \(\operatorname{Osc}_1(f,r)\) is nondecreasing in \(r\).
\end{enumerate}
\end{proposition}

We now recall a key estimate that will be used repeatedly.

\begin{proposition}\label{prop:LY-estimate}
Let \(\tau : I \to I\) be monotone on an interval \(J \subset I\), and suppose that \(|\tau'(x)| \ge s > 1\) for all \(x \in J\). Then for any function \(f : J \to \mathbb{R}\) and any \(r > 0\),
\begin{equation}\label{eq:osc-inverse}
    \operatorname{Osc}(f \circ \tau^{-1}, r, y)
    \le
    \operatorname{Osc}\left(f, \frac{r}{s}, \tau^{-1}(y)\right)
\end{equation}
for all \(y \in \tau(J)\).
\end{proposition}

This estimate reflects the contraction of inverse branches of expanding maps and plays a crucial role in establishing Lasota--Yorke type inequalities.

Finally, we define the class of maps under consideration.

\begin{definition}\label{def:T-class}
    A map \(\tau: I \to I\) belongs to the class \(\mathcal{T}([0,1]; s, \varepsilon)\) if there exists a finite partition \(0 = a_0 < a_1 < \cdots < a_q = 1\) such that:
\begin{enumerate}
    \item for each \(i=1,\dots,q\), \(\tau\) is monotone and \(C^{1+\varepsilon}\) on \((a_{i-1}, a_i)\), extending continuously to \([a_{i-1}, a_i]\);
    \item \(|\tau'(x)| \geq s > 1\) for all \(x\) where defined;
    \item \(\tau'\) is \(\varepsilon\)-Hölder continuous on each interval.
\end{enumerate}
\end{definition}

\begin{assumption}\label{ass:LY}[Lasota--Yorke condition for \(P_\tau\)]
 By Theorem~3.1 of~\cite{RajputGora2025}, there exist constants $0 < \alpha_0 < 1$ and
  $C_0 > 0$, depending only on $\tau$, such that
  \[
    \operatorname{var}_{1,1/p}(P_\tau f) \leq \alpha_0 \|f\|_{1,1/p} + C_0 \|f\|_1
    \qquad \text{for all } f \in BV_{1,1/p}(I).
  \]
  Moreover, $\alpha_0 \leq s^{-1/p}$ by~\cite[Theorem~3.1]{RajputGora2025}.
\end{assumption}

These preliminaries will be used in the subsequent sections to analyze the perturbed operators and establish stochastic stability.
\section{Perturbation Operators and Uniform Lasota--Yorke Inequality}

  Let $q : \mathbb{R} \to [0,\infty)$ be a $C^1$ function with compact support satisfying
  \[
      \int_{\mathbb{R}} q(z)\,dz = 1, \qquad \operatorname{supp}(q) \subset [-1,1],
  \]
  and assume in addition that $q' \in L^\infty(\mathbb{R})$. For $\delta \in (0,1/4)$, define
  the \emph{periodized kernel}
  \[
      q_\delta^{\mathrm{per}}(x,y)
      = \sum_{n\in\mathbb{Z}} \delta^{-1}\,q\!\left(\frac{x-y+n}{\delta}\right),
      \qquad x,y \in I,
  \]
  and the \emph{smoothing operator}
  \[
      (Q_\delta f)(x) = \int_I q_\delta^{\mathrm{per}}(x,y)\,f(y)\,dy, \qquad x \in I.
  \]
  \begin{assumption}\label{ass:joint}
  We assume that $\tau$ and $q$ satisfy the joint expansion--kernel condition
  \[
    \alpha_0 \;<\; \frac{1}{\widetilde{C}_1},
    \qquad
    \widetilde{C}_1 := (1+2^{1+1/p})\cdot 2^{1/p}\max\!\bigl\{1,\,4\|q'\|_\infty\bigr\}.
  \]
  Since $\alpha_0 \leq s^{-1/p}$ (Assumption~\ref{ass:LY}), this holds whenever
  $s^{1/p} > \widetilde{C}_1$, i.e., $s > \widetilde{C}_1^{\,p}$.
  For any fixed $\tau$, it is satisfied by choosing $q$ with $\|q'\|_\infty$ small enough
  that $4\|q'\|_\infty \leq 1$ and $s^{1/p} > (1+2^{1+1/p})\cdot 2^{1/p}$.
  \end{assumption}
  Since $\operatorname{supp}(q)\subset[-1,1]$ and $\delta<1/4$, for each pair $(x,y)\in I\times I$
  at most two terms in the sum are nonzero. The substitution $u=(x-y+n)/\delta$ gives
  \[
      \int_I q_\delta^{\mathrm{per}}(x,y)\,dx
      = \sum_{n\in\mathbb{Z}}\int_{(n-y)/\delta}^{(n+1-y)/\delta} q(u)\,du
      = \int_{-\infty}^\infty q(u)\,du = 1
      \quad\text{for all } y\in I,
  \]
  and the identical argument with the roles of $x$ and $y$ interchanged gives
  $\int_I q_\delta^{\mathrm{per}}(x,y)\,dy=1$ for all $x\in I$, so the kernel is
  doubly stochastic on $I$. For $x\in[\delta,1-\delta]$ the support of
  $q_\delta^{\mathrm{per}}(x,\cdot)$ lies entirely inside $I$, so
  $q_\delta^{\mathrm{per}}(x,y)=\delta^{-1}q((x-y)/\delta)$ there; the periodization
  affects only the strips $[0,\delta)\cup(1-\delta,1]$. Since $q_\delta^{\rm per}$ is precisely the convolution kernel of $\mathbb{T} = \mathbb{R}/\mathbb{Z}$
  restricted to $I \times I$, the operator $Q_\delta$ on $I$ coincides with torus convolution.
  We identify functions on $I$ with their periodic extensions to $\mathbb{T}$ when applying
  convolution operators. The perturbed operator is $P_\delta = Q_\delta P_\tau$.
  The oscillation estimates established in Lemmas~\ref{lem:osc-small-r} and ~\ref{lem:translation} below therefore apply directly
  to $Q_\delta$ on $I$ without any boundary correction.

  \begin{proposition}\label{prop:markov}
  For each $\delta>0$, the operator $P_\delta$ is positive, preserves integrals,
  \[
      \int_I P_\delta f\,dm = \int_I f\,dm \quad\text{for all } f\in L^1(I),
  \]
  and is a contraction on $L^1(I)$: $\|P_\delta f\|_1 \le \|f\|_1$ for all $f\in L^1(I)$.
  Moreover, if $f\ge 0$ then $\|P_\delta f\|_1 = \|f\|_1$.
  \end{proposition}

  \begin{proof}
  The operators $P_\tau$ and $Q_\delta$ are positive and linear. The Frobenius--Perron
  operator $P_\tau$ preserves integrals by construction.

  For $Q_\delta$, since $q_\delta^{\mathrm{per}}\ge 0$, Tonelli's theorem gives
  \[
      \int_I\!\int_I q_\delta^{\mathrm{per}}(x,y)|f(y)|\,dm(y)\,dm(x)
      = \int_I |f(y)|\underbrace{\left[\int_I q_\delta^{\mathrm{per}}(x,y)\,dm(x)\right]}_{=\,1}
        dm(y) = \|f\|_1 < \infty,
  \]
  so the double integral is absolutely convergent. Fubini's theorem then permits
  interchange of integration order for $f$ itself:
  \[
      \int_I (Q_\delta f)(x)\,dm(x)
      = \int_I f(y)\left[\int_I q_\delta^{\mathrm{per}}(x,y)\,dm(x)\right]dm(y)
      = \int_I f(y)\,dm(y).
  \]
  Hence $Q_\delta$ preserves integrals, and so does $P_\delta = Q_\delta P_\tau$.

  For the $L^1(I)$ bound, positivity of $P_\delta$ and linearity give
  $|P_\delta f|\le P_\delta|f|$ pointwise. Integrating and applying integral preservation
  (with $|f|\ge 0$):
  \[
      \|P_\delta f\|_1
      = \int_I|P_\delta f|\,dm \le \int_I P_\delta|f|\,dm = \int_I|f|\,dm = \|f\|_1.
  \]
  If $f\ge 0$, then $|P_\delta f|=P_\delta f$, so equality holds. Thus \(Q_\delta\), and hence \(P_\delta\), is a Markov operator.
  \end{proof}

  \begin{lemma}[Small-scale oscillation on $\mathbb{T}$]\label{lem:osc-small-r}
  Work on $\mathbb{T}=\mathbb{R}/\mathbb{Z}$ with $0<\delta<1/2$. Let
  $q_\delta(x,y)=\tfrac{1}{\delta}q\!\left(\tfrac{x-y}{\delta}\right)$,  where $q \in C^1$, $\operatorname{supp}(q) \subset [-1, 1]$, $q \geq 0$,
  $\int_{\mathbb{R}} q(t)\,dt = 1$. Set
  $C:=4\|q'\|_\infty$. For all $g\in L^1(\mathbb{T})$ and $0<r\le\delta$,
  \[
      \operatorname{Osc}_1(Q_\delta g,\,r) \;\le\; \frac{Cr}{\delta}\,\operatorname{Osc}_1(g,\,r+\delta).
  \]
  \end{lemma}

  \begin{proof}
  Fix $r>0$ and $x\in\mathbb{T}$. Let $x_1,x_2\in B(x,r)$, so $|x_1-x_2|\le 2r$.
  The kernels $q_\delta(x_i,\cdot)$ are supported on
  $[x_i-\delta,x_i+\delta]\subset B(x,r+\delta)$ (since $|x_i-x|\le r$).

  Set $c:=\frac{1}{|B(x,r+\delta)|}\int_{B(x,r+\delta)} g(y)\,dy$.
  Since $\int q_\delta(x_i,y)\,dy=1$ for each $i$,
  \[
      (Q_\delta g)(x_1)-(Q_\delta g)(x_2)
      = \int_{\mathbb{T}}
        \bigl(q_\delta(x_1,y)-q_\delta(x_2,y)\bigr)\bigl(g(y)-c\bigr)\,dy.
  \]
  The integrand is supported inside $B(x,r+\delta)$. For $y\in B(x,r+\delta)$,
  $|g(y)-c|\le\operatorname{Osc}(g,r+\delta,x)$.
  For the kernel difference, substituting $u=(x_1-y)/\delta$ and using the mean value theorem and the fact that \(\operatorname{supp}(q')\subset[-1,1]\), so that \(\|q'\|_{1}\le 2\|q'\|_\infty\), we obtain
\[
\int_{\mathbb{T}} |q_\delta(x_1,y)-q_\delta(x_2,y)|\,dy
\le 2\|q'\|_\infty \cdot \frac{|x_1-x_2|}{\delta}
\le \frac{Cr}{\delta},
\]
where \(C=4\|q'\|_\infty\)
  and $|x_1-x_2|\le 2r$. Combining,
  $\operatorname{Osc}(Q_\delta g,r,x)\le\frac{Cr}{\delta}\operatorname{Osc}(g,r+\delta,x)$.
  Integrating over $x\in\mathbb{T}$ gives the result.
  \end{proof}

  \begin{lemma}[Translation coupling on $\mathbb{T}$]\label{lem:translation}
  Work on $\mathbb{T}$. For all $g\in L^1(\mathbb{T})$ and $r>0$,
  \[
      \operatorname{Osc}_1(Q_\delta g,\,r) \;\le\; \operatorname{Osc}_1(g,\,r+\delta).
  \]
  \end{lemma}

  \begin{proof}
  Fix $x\in\mathbb{T}$ and $x_1,x_2\in B(x,r)$. Set $v:=x_2-x_1$.
 Since \(q_\delta(x_2,y)=q_\delta(x_1,y-v)\), by translation invariance on \(\mathbb T\), the change of variables \(u=y-v\) gives
\[
(Q_\delta g)(x_2)=\int_{\mathbb T} q_\delta(x_1,u)\,g(u+v)\,du.
\]
  and therefore
  \[
      (Q_\delta g)(x_1)-(Q_\delta g)(x_2)
      = \int_\mathbb{T} q_\delta(x_1,u)\bigl[g(u)-g(u+v)\bigr]\,du.
  \]
  The kernel $q_\delta(x_1,\cdot)$ is supported on
  $[x_1-\delta,x_1+\delta]\subset B(x,r+\delta)$. For each $u$ in this support:
  $u\in B(x,r+\delta)$, and $u+v\in[x_2-\delta,x_2+\delta]\subset B(x,r+\delta)$
  (since $|(u+v)-x_2|=|u-x_1|\le\delta$). Hence
  $|g(u)-g(u+v)|\le\operatorname{Osc}(g,r+\delta,x)$.
  Since $\int q_\delta(x_1,u)\,du=1$,
  \[
      |(Q_\delta g)(x_1)-(Q_\delta g)(x_2)|\le\operatorname{Osc}(g,r+\delta,x).
  \]
  Taking the supremum over $x_1,x_2\in B(x,r)$ and integrating over $x\in\mathbb{T}$
  gives the result.
  \end{proof}
\begin{remark}
When it is necessary to distinguish integration over the interval
$I = [0,1]$ from integration over the torus $\mathbb{T} = \mathbb{R}/\mathbb{Z}$,
we write $\mathrm{Osc}_1^I(f, r)$ in place of $\mathrm{Osc}_1(f, r)$;
the superscript indicates that the $L^1$-integration in \eqref{eq:osc1}
is taken over $I$.
The unsuperscripted notation $\mathrm{Osc}_1(f,r)$ is used when the ambient
space is clear from context.
Explicit superscripts $I$ and $\mathbb{T}$ are introduced when both domains
are considered simultaneously.
In particular, equation~\eqref{eq:BVnorm} uses $\mathrm{Osc}_1^I$
to emphasize integration over the interval.
\end{remark}
  \begin{lemma}[Interval-to-torus variation comparison]\label{lem:I-to-T}
  For any $g\in BV_{1,1/p}(I)$,
  \[
      \operatorname{var}_{1,1/p}^{\,\mathbb{T}}(g^{\mathbb{T}})
      \;\le\; (1+2^{1+1/p})\,\operatorname{var}_{1,1/p}^{\,I}(g),
  \]
  where $g^{\mathbb{T}}$ denotes the $\mathbb{Z}$-periodization of $g|_I$.
  \end{lemma}

  \begin{proof}
   It suffices to consider $r \in (0, 1/2]$. For $r \geq 1/2$, the torus ball satisfies
  $B^{\mathbb{T}}(x, r) = \mathbb{T}$ for every $x \in \mathbb{T}$, so
  $\operatorname{Osc}_1^{\mathbb{T}}(g^{\mathbb{T}}, r)$ is constant on $[1/2, 1)$.
  Since $r^{-1/p}$ is strictly decreasing in $r$, the supremum defining
  $\operatorname{var}_{1,1/p}^{\mathbb{T}}(g^{\mathbb{T}})$ is therefore achieved at some
  $r \in (0, 1/2]$, and it suffices to bound $r^{-1/p}\operatorname{Osc}_1^{\mathbb{T}}(g^{\mathbb{T}},r)$
  for $r \leq 1/2$. For $r\le 1/2$,
  split the torus oscillation integral:
  \begin{equation}\label{eq:torus-split}
      \operatorname{Osc}_1^{\,\mathbb{T}}(g^{\mathbb{T}},r)
      = \int_{[r,1-r]}\operatorname{Osc}^I(g,r,x)\,dx
        + \int_{[0,r)\cup(1-r,1]}\operatorname{Osc}^{\mathbb{T}}(g^{\mathbb{T}},r,x)\,dx.
  \end{equation}
  The first integral is bounded by $\operatorname{Osc}_1^I(g,r)$.

  For the boundary strips, if $x\in[0,r)$ the torus ball satisfies
  \[
      B^{\mathbb{T}}(x,r)=[0,x+r]\cup[1-r+x,1]
      \subset B^I(0,2r)\cup B^I(1,2r).
  \]
  Since \(B^{\mathbb T}(x,r)\subset B^I(0,2r)\cup B^I(1,2r)\), oscillation on the boundary strip is controlled by the oscillations near the two endpoints,
hence \(\operatorname{Osc}^{\mathbb{T}}(g^{\mathbb{T}},r,x)
\le\operatorname{Osc}^I(g,2r,0)+\operatorname{Osc}^I(g,2r,1)\).
  The same bound holds for $x\in(1-r,1]$. Integrating over the boundary strips:
  \[
      \int_{[0,r)\cup(1-r,1]}\operatorname{Osc}^{\mathbb{T}}(g^{\mathbb{T}},r,x)\,dx
      \le 2r\bigl[\operatorname{Osc}^I(g,2r,0)+\operatorname{Osc}^I(g,2r,1)\bigr].
  \]
  For $x_0\in\{0,1\}$, if $x_1,x_2\in B^I(x_0,2r)$ and $x\in B^I(x_0,2r)$, then
  $|x-x_1|,|x-x_2|\le 2r$, hence $|g(x_1)-g(x_2)|\le\operatorname{Osc}^I(g,2r,x)$ a.e.
  Integrating over $x\in B^I(x_0,2r)$:
  \[
      2r\cdot\operatorname{Osc}^I(g,2r,x_0)
      \le\int_{B^I(x_0,2r)}\operatorname{Osc}^I(g,2r,x)\,dx
      \le\operatorname{Osc}_1^I(g,2r).
  \]
  Therefore,
  \[
      \int_{[0,r)\cup(1-r,1]}\operatorname{Osc}^{\mathbb{T}}(g^{\mathbb{T}},r,x)\,dx
      \le 2\,\operatorname{Osc}_1^I(g,2r).
  \]
  Substituting into \eqref{eq:torus-split} and multiplying by $r^{-1/p}$:
  \begin{align*}
      r^{-1/p}\operatorname{Osc}_1^{\,\mathbb{T}}(g^{\mathbb{T}},r)
      &\le r^{-1/p}\operatorname{Osc}_1^I(g,r)+2r^{-1/p}\operatorname{Osc}_1^I(g,2r) \\
      &\le \operatorname{var}_{1,1/p}^I(g)
           +2\cdot 2^{1/p}(2r)^{-1/p}\operatorname{Osc}_1^I(g,2r) \\
      &\le \operatorname{var}_{1,1/p}^I(g)+2^{1+1/p}\operatorname{var}_{1,1/p}^I(g) \\
      &= (1+2^{1+1/p})\operatorname{var}_{1,1/p}^I(g).
  \end{align*}
  Taking the supremum over $r\in(0,1/2]$ gives the result.
  \end{proof}

  \begin{theorem}[Uniform Lasota--Yorke inequality]\label{thm:uniform-LY}
  There exist constants $0<\alpha<1$ and $C>0$, independent of $\delta>0$, such that
  for all $f\in BV_{1,1/p}(I)$,
  \[
      \|P_\delta f\|_{1,1/p}
      \;\le\; \alpha\|f\|_{1,1/p} + C\|f\|_{1}.
  \]
  \end{theorem}

  \begin{proof}
  Let $g:=P_\tau f$, so $P_\delta f=Q_\delta g$ and $\|g\|_{1}=\|f\|_{1}$.
  By Assumption~\ref{ass:LY}:
  \begin{equation}\label{eq:LY-Ptau}
       \operatorname{var}^I_{1,1/p}(g) \leq \alpha_0\|f\|_{1,1/p} + C_0\|f\|_1.
  \end{equation}
  We estimate
  $\operatorname{var}_{1,1/p}^I(Q_\delta g)
  =\sup_{0<r\le 1} r^{-1/p}\operatorname{Osc}_1^I(Q_\delta g,r)$.

  Since $Q_\delta$ on $I$ equals the torus convolution $Q_\delta^{\mathbb{T}}$ restricted
  to $I$, we have $(Q_\delta g)(x)=(Q_\delta^{\mathbb{T}}g^{\mathbb{T}})(x)$ for all
  $x\in I$, and therefore
  \begin{equation}\label{eq:I-le-T}
      \operatorname{Osc}_1^I(Q_\delta g,r)
      \;\le\; \operatorname{Osc}_1^{\,\mathbb{T}}(Q_\delta^{\mathbb{T}}g^{\mathbb{T}},r).
  \end{equation}
  We split into two cases.

  \textbf{Case 1: $r\ge\delta$.} By Lemma~\ref{lem:translation} applied on $\mathbb{T}$:
  \[
      \operatorname{Osc}_1^{\,\mathbb{T}}(Q_\delta^{\mathbb{T}}g^{\mathbb{T}},r)
      \le\operatorname{Osc}_1^{\,\mathbb{T}}(g^{\mathbb{T}},r+\delta)
      \le (r+\delta)^{1/p}\operatorname{var}_{1,1/p}^{\,\mathbb{T}}(g^{\mathbb{T}}).
  \]
  Since $r+\delta\le 2r$,
  \[
      r^{-1/p}\operatorname{Osc}_1^I(Q_\delta g,r)
      \;\le\; 2^{1/p}\operatorname{var}_{1,1/p}^{\,\mathbb{T}}(g^{\mathbb{T}}).
  \]

  \textbf{Case 2: $0<r<\delta$.} By Lemma~\ref{lem:osc-small-r} applied on $\mathbb{T}$:
  \[
      \operatorname{Osc}_1^{\,\mathbb{T}}(Q_\delta^{\mathbb{T}}g^{\mathbb{T}},r)
      \le\frac{Cr}{\delta}(r+\delta)^{1/p}
         \operatorname{var}_{1,1/p}^{\,\mathbb{T}}(g^{\mathbb{T}}).
  \]
  Thus
\[
    r^{-1/p}\operatorname{Osc}_1^I(Q_\delta g,r)
    \le C\, r^{1-1/p}\delta^{-1}(r+\delta)^{1/p}
       \operatorname{var}_{1,1/p}^{\,\mathbb{T}}(g^{\mathbb{T}}).
\]
  The function $\varphi(r)=r^{1-1/p}(r+\delta)^{1/p}$ is increasing on $(0,\delta)$,
  so $\sup_{0<r<\delta}\varphi(r)=\varphi(\delta)=2^{1/p}\delta$. Hence
  \[
      r^{-1/p}\operatorname{Osc}_1^I(Q_\delta g,r)
      \;\le\; C\cdot 2^{1/p}\operatorname{var}_{1,1/p}^{\,\mathbb{T}}(g^{\mathbb{T}}).
  \]

  Combining both cases:
  \[
      \operatorname{var}_{1,1/p}^I(Q_\delta g)
      \;\le\; C_1\operatorname{var}_{1,1/p}^{\,\mathbb{T}}(g^{\mathbb{T}}),
  \]
  where \(C_1=2^{1/p}\max\{1,4\|q'\|_\infty\}\), coming from Lemmas~\ref{lem:osc-small-r} and ~\ref{lem:translation}.
  Now apply Lemma~\ref{lem:I-to-T}:
  \[
      \operatorname{var}_{1,1/p}^I(Q_\delta g)
      \;\le\; C_1(1+2^{1+1/p})\operatorname{var}_{1,1/p}^I(g).
  \]
   Recalling $\widetilde{C}_1 = C_1(1 + 2^{1+1/p})$ from Assumption~\ref{ass:joint}. By~(3.2) and
  the $L^1$-contraction $\|P_\delta f\|_1 \leq \|f\|_1$:
  \[
      \|P_\delta f\|_{1,1/p}
      \;\le\; \widetilde{C}_1\alpha_0\|f\|_{1,1/p}
              + (\widetilde{C}_1 C_0+1)\|f\|_{1}.
  \]
  Let $\alpha:=\widetilde{C}_1\alpha_0$ and $C:=\widetilde{C}_1 C_0+1$.
  Both constants are independent of $\delta$. Since
  $\alpha_0<1/\widetilde{C}_1$ (Assumption~\ref{ass:joint}), it follows that
  $\alpha=\widetilde{C}_1\alpha_0\in(0,1)$.
  \end{proof}

  This uniform estimate is the key ingredient ensuring uniform quasi-compactness of
  the family $(P_\delta)_{\delta>0}$.
  \section{Existence and Convergence of Invariant Densities}

  In this section, we prove the existence of invariant densities for the perturbed
  operators $P_\delta$ and establish their convergence in $L^1(I)$ to the invariant density
  of the unperturbed system.

  \begin{theorem}\label{thm:existence}
  For each $\delta>0$, the operator $P_\delta$ admits an invariant density
  $h_\delta\in BV_{1,1/p}(I)$ such that
  \[
      P_\delta h_\delta = h_\delta, \qquad h_\delta\ge 0, \qquad \|h_\delta\|_1=1.
  \]
  \end{theorem}

  \begin{proof}
  Let $f\in BV_{1,1/p}(I)$ with $f\ge 0$ and $\|f\|_1=1$.
  By Theorem~\ref{thm:uniform-LY}, there exist constants $\alpha\in(0,1)$ and $C>0$,
  independent of $\delta$, such that
  \[
      \|P_\delta g\|_{1,1/p}
      \le \alpha\|g\|_{1,1/p} + C\|g\|_1
      \quad\text{for all }g\in BV_{1,1/p}(I).
  \]
   Iterating and using $\|P^k_\delta f\|_1 = \|f\|_1 = 1$ (since $f \geq 0$ and $P_\delta$
  is Markov, so it preserves the $L^1(I)$ norm of nonnegative functions), we obtain for all $n\in\mathbb{N}$:
  \[
      \|P_\delta^n f\|_{1,1/p}
      \le \alpha^n\|f\|_{1,1/p} + C\sum_{k=0}^{n-1}\alpha^k\|f\|_1
      \le \|f\|_{1,1/p} + \frac{C}{1-\alpha}
      =: C'.
  \]
  Define the Ces\`aro averages $h_n:=\frac{1}{n}\sum_{k=0}^{n-1}P_\delta^k f$.
  By convexity of the norm, $\|h_n\|_{1,1/p}\le C'$, so $\{h_n\}$ is bounded in
  $BV_{1,1/p}(I)$. Since the embedding $BV_{1,1/p}(I)\hookrightarrow L^1(I)$ is compact,
  there exist a subsequence $\{h_{n_j}\}$ and $h_\delta\in L^1(I)$ such that
  $h_{n_j}\to h_\delta$ in $L^1(I)$.

  By lower semicontinuity of $\operatorname{var}_{1,1/p}$ with respect to $L^1(I)$
  convergence \cite{keller1985},
  \[
      \operatorname{var}_{1,1/p}(h_\delta)
      \le \liminf_{j\to\infty}\operatorname{var}_{1,1/p}(h_{n_j}) \le C',
  \]
  so $h_\delta\in BV_{1,1/p}(I)$.

  To verify invariance, observe the telescoping identity
  \[
      P_\delta h_n - h_n
      = \frac{1}{n}\sum_{k=0}^{n-1}(P_\delta^{k+1}f - P_\delta^k f)
      = \frac{1}{n}(P_\delta^n f - f).
  \]
  Taking $L^1(I)$ norms and using $\|P_\delta^n f\|_1=\|f\|_1=1$:
  \[
      \|P_\delta h_n - h_n\|_1 \le \frac{2}{n} \to 0.
  \]
  Since $P_\delta:L^1(I)\to L^1(I)$ is a bounded linear operator,
  $h_{n_j}\to h_\delta$ in $L^1(I)$ implies $P_\delta h_{n_j}\to P_\delta h_\delta$
  in $L^1(I)$. Combined with $\|P_\delta h_{n_j}-h_{n_j}\|_1\to 0$ and
  $h_{n_j}\to h_\delta$, we conclude $P_\delta h_\delta = h_\delta$.

  Finally, since $h_{n_j}\to h_\delta$ in $L^1(I)$, there is a further subsequence along
  which $h_{n_j}(x)\to h_\delta(x)$ a.e. Since each $h_{n_j}\ge 0$, we get
  $h_\delta\ge 0$ a.e. Moreover,
  $\|h_\delta\|_1=\lim_{j\to\infty}\|h_{n_j}\|_1=1$.
  \end{proof}

  \begin{lemma}\label{lem:Qdelta-to-id}
Let \(K\subset L^1(I)\) be relatively compact. Then
\[
\sup_{g\in K}\|Q_\delta g-g\|_{1}\to 0
\qquad\text{as }\delta\to 0.
\]
\end{lemma}

\begin{proof}
Since \(K\) is relatively compact in \(L^1(I)\), it is totally bounded. Hence, for any
\(\varepsilon>0\), there exist \(g_1,\ldots,g_N\in L^1(I)\) such that for every \(g\in K\)
there exists \(i\) with
\[
\|g-g_i\|_{1}<\varepsilon.
\]

For any \(g\in K\), choose such \(g_i\). Then
\[
\|Q_\delta g-g\|_{1}
\le \|Q_\delta g-Q_\delta g_i\|_{1}
   +\|Q_\delta g_i-g_i\|_{1}
   +\|g_i-g\|_{1}.
\]
Since \(Q_\delta\) is a Markov operator, it is a contraction on \(L^1(I)\), so
\[
\|Q_\delta g-Q_\delta g_i\|_{1}\le \|g-g_i\|_{1}<\varepsilon.
\]
Thus,
\[
\sup_{g\in K}\|Q_\delta g-g\|_{1}
\le 2\varepsilon + \max_{1\le i\le N}\|Q_\delta g_i-g_i\|_{1}.
\]

It remains to show that for each fixed \(i\),
\[
\|Q_\delta g_i-g_i\|_{1}\to 0
\qquad\text{as }\delta\to 0.
\]

Let \(g_i^{\mathbb T}\) denote the periodic extension of \(g_i\) to
\(\mathbb T\). Since \(Q_\delta\) coincides with convolution on \(\mathbb T\)
(via periodic extension), we have
\[
(Q_\delta g_i)(x)=(Q_\delta^{\mathbb T}g_i^{\mathbb T})(x),
\qquad x\in I,
\]
where $Q_\delta^{\mathbb{T}}$ is the torus convolution with kernel
  $q_\delta(x,y) = \frac{1}{\delta}q\!\left(\frac{x-y}{\delta}\right)$ on $\mathbb{T}$.
  Because \(q_\delta\) is the periodized convolution kernel with total mass \(1\) and support shrinking to a point as \(\delta\to 0\),  it forms an approximate identity on $\mathbb{T}$; it follows that
\[
\|Q_\delta^{\mathbb T}g_i^{\mathbb T}-g_i^{\mathbb T}\|_{L^1(\mathbb T)}\to 0.
\]
Restricting to \(I\), we obtain
\[
\|Q_\delta g_i-g_i\|_{1}
\le \|Q_\delta^{\mathbb T}g_i^{\mathbb T}-g_i^{\mathbb T}\|_{L^1(\mathbb T)}
\to 0.
\]

Therefore,
\[
\max_{1\le i\le N}\|Q_\delta g_i-g_i\|_{1}\to 0,
\]
and hence
\[
\limsup_{\delta\to 0}\sup_{g\in K}\|Q_\delta g-g\|_{1}\le 2\varepsilon.
\]
Since \(\varepsilon>0\) is arbitrary, the result follows.
\end{proof}

  \begin{theorem}[Stochastic stability]\label{thm:stochastic-stability}
  Assume $\tau\in\mathcal{T}([0,1];s,\varepsilon)$ is topologically mixing, so that
  $P_\tau$ admits a unique ACIM with density $h\in BV_{1,1/p}(I)$. Let $h_\delta$ be the
  invariant density of $P_\delta$ given by Theorem~\ref{thm:existence}. Then
  $\|h_\delta-h\|_1\to 0$ as $\delta\to 0$.
  \end{theorem}

  \begin{proof}
  We define an equivalent norm on $BV_{1,1/p}(I)$ by
  \[
      \|f\|_{\mathcal{B}} := \operatorname{var}_{1,1/p}(f) + A\|f\|_1,
      \qquad A>0\text{ chosen below},
  \]
  as the strong norm, and retain $\|\cdot\|_1$ as the weak norm.

  \textbf{Step 1: Uniform Lasota--Yorke inequality and operator bound.}
  By Theorem~\ref{thm:uniform-LY}, there exist $\alpha\in(0,1)$ and $C>0$,
  independent of $\delta$, such that
  \[
      \operatorname{var}_{1,1/p}(P_\delta f)
      \le \alpha\operatorname{var}_{1,1/p}(f) + C\|f\|_1,
      \qquad
      \|P_\delta f\|_1 \le \|f\|_1.
  \]
  Choose $A>C/(1-\alpha)$. Adding $A$ times the second inequality to the first:
  \[
      \|P_\delta f\|_{\mathcal{B}}
      \le \alpha\|f\|_{\mathcal{B}} + B\|f\|_{1},
      \qquad B:=C+A(1-\alpha),
  \]
  uniformly in $\delta>0$. Since $\|f\|_{1}\le A^{-1}\|f\|_{\mathcal{B}}$:
  \[
  \|P_\delta f\|_{\mathcal{B}} \leq M\|f\|_{\mathcal{B}}, \quad
  M := \alpha + \frac{B}{A} = \alpha + \frac{C}{A} + (1-\alpha) = 1 + \frac{C}{A} < \infty,
  \]
  so $\sup_{\delta>0}\|P_\delta\|_{\mathcal{B}\to\mathcal{B}}\le M$.

  \textbf{Step 2: Perturbation smallness
  $\|P_\delta-P_\tau\|_{\mathcal{B}\to L^1(I)}\to 0$.}
  Write $P_\delta f - P_\tau f = (Q_\delta-I)P_\tau f$. Since $P_\tau$ satisfies the
  same Lasota--Yorke inequality \cite{RajputGora2025}, it maps the unit ball of
  $(BV_{1,1/p}(I),\|\cdot\|_{\mathcal{B}})$ into a set bounded in $BV_{1,1/p}(I)$.
  By compactness of $BV_{1,1/p}(I)\hookrightarrow L^1(I)$, the set
  \[
      K := P_\tau\bigl(\{\|f\|_{\mathcal{B}}\le 1\}\bigr)
  \]
  is relatively compact in $L^1(I)$. Lemma~\ref{lem:Qdelta-to-id} then gives
  \[
      \sup_{\|f\|_{\mathcal{B}}\le 1}\|(Q_\delta-I)P_\tau f\|_1
      = \sup_{g\in K}\|Q_\delta g-g\|_1 \to 0,
  \]
  and hence $\|P_\delta-P_\tau\|_{\mathcal{B}\to L^1(I)}\to 0$.

  \textbf{Step 3: Spectral gap for $P_\tau$.}

  \textit{Quasi-compactness.} By the Lasota--Yorke inequality for $P_\tau$ and the
  compact embedding $BV_{1,1/p}(I)\hookrightarrow L^1(I)$, the Ionescu--Tulcea--Marinescu
  theorem \cite{ionescu-tulcea-marinescu} implies that $P_\tau$ is quasi-compact on
  $(BV_{1,1/p}(I),\|\cdot\|_{\mathcal{B}})$, with essential spectral radius
  $r_{\mathrm{ess}}(P_\tau)\le\alpha<1$.

  \textit{Simplicity of eigenvalue $1$.} Let $g\in BV_{1,1/p}(I)$ with $P_\tau g=g$.
  Write $g=g^+-g^-$ with $g^\pm\ge 0$. By positivity of $P_\tau$:
  $P_\tau|g|\ge|P_\tau g|=|g|$.
  Since $P_\tau$ is a Markov operator, $\int P_\tau|g|\,dm=\int|g|\,dm$.
  Since $P_\tau|g|\ge|g|$ a.e. and both functions have the same integral,
  $P_\tau|g|=|g|$ a.e. Adding this to $P_\tau g=g$:
  \[
      2P_\tau g^+ = P_\tau|g|+P_\tau g = 2g^+,
  \]
  so $P_\tau g^+=g^+$ a.e., and likewise $P_\tau g^-=g^-$ a.e.
  If $g^+\not\equiv 0$, then $g^+/\|g^+\|_1$ is a normalized nonnegative fixed point
  of $P_\tau$; by uniqueness of the ACIM (standard under topological mixing~\cite{baladi2000,hofbauer1982};
  see Remark~1 below),
  $g^+=c^+h$. Similarly $g^-=c^-h$. Hence $g=(c^+-c^-)h$, and
  $\ker(P_\tau-I)$ is one-dimensional.

  \textit{Isolation.} Since $r_{\mathrm{ess}}(P_\tau)\le\alpha<1$, all eigenvalues of
  $P_\tau$ on the unit circle are isolated with finite multiplicity.
  The eigenvalue $1$ is simple by the above, hence isolated. Its spectral projection
  is $\Pi f=\bigl(\int f\,dm\bigr)h$, which has rank one.

  \textbf{Step 4: Application of Keller--Liverani.}
  Steps 1--3 verify the hypotheses of \cite[Theorem~1]{keller-liverani}:
  \begin{enumerate}
  \item Uniform Lasota--Yorke: $\|P_\delta f\|_{\mathcal{B}}
        \le\alpha\|f\|_{\mathcal{B}}+B\|f\|_{1}$, with $\alpha\in(0,1)$ and
        $B<\infty$ independent of $\delta$.
  \item Perturbation smallness:
        $\|P_\delta-P_\tau\|_{\mathcal{B}\to L^1(I)}\to 0$.
  \item Spectral isolation: eigenvalue $1$ of $P_\tau$ is isolated and simple,
        with rank-one spectral projection $\Pi$.
  \end{enumerate}
  The Keller--Liverani theorem therefore yields spectral projections $\Pi_\delta$
  (rank one for all sufficiently small $\delta$) satisfying
  \[
      \|\Pi_\delta-\Pi\|_{\mathcal{B}\to L^1(I)}\to 0 \quad\text{as }\delta\to 0.
  \]
  Applying this to $h\in BV_{1,1/p}(I)$:
  \[
      \|\Pi_\delta h-h\|_1
      \le \|\Pi_\delta-\Pi\|_{\mathcal{B}\to L^1(I)}\|h\|_{\mathcal{B}}\to 0,
  \]
  and $\|\Pi_\delta h\|_1\to 1$. The unique normalized invariant density of $P_\delta$ is
  \[
      h_\delta = \frac{\Pi_\delta h}{\|\Pi_\delta h\|_1},
  \]
  and therefore $\|h_\delta-h\|_1\to 0$.

  It remains to observe that the invariant density produced here coincides with that of
  Theorem~\ref{thm:existence}. Indeed, Keller--Liverani implies that for all sufficiently
  small $\delta>0$, the eigenspace of $P_\delta$ at eigenvalue $1$ is one-dimensional.
  Hence any nonnegative normalized fixed point of $P_\delta$ in $BV_{1,1/p}(I)$ must equal
  $\Pi_\delta h/\|\Pi_\delta h\|_1$, so the two constructions agree.
  \end{proof}

  \begin{remark}\label{rem:mixing}
  Topological mixing holds for $\tau\in\mathcal{T}([0,1];s,\varepsilon)$ whenever the
  Markov graph of the partition $\{a_0,\ldots,a_q\}$ is aperiodic and strongly connected,
 in particular when every branch is full (i.e.\ $\tau([a_{i-1}, a_i]) = [0,1]$ for all $i$).
  See~\cite{hofbauer1982}.
  \end{remark}

\section{Example: Nonlinear Piecewise Expanding Map}

We present a nonlinear example illustrating the applicability of our results beyond piecewise linear maps.

\begin{example}[Nonlinear piecewise expanding map]\label{ex:nonlinear}
Let $\eta > 0$ be such that $\eta < \frac{1}{2\pi}$, and define
\[
F(x) := 2x + \eta \sin(2\pi x), \qquad x \in [0,1].
\]
Then
\[
F'(x) = 2 + 2\pi \eta \cos(2\pi x) \ge 2 - 2\pi \eta > 1,
\]
so $F$ is strictly increasing on $[0,1]$. Since $F(0)=0$ and $F(1)=2$, by the intermediate value theorem there exists a unique $a \in (0,1)$ such that
\[
F(a) = 1.
\]

Define $\tau : [0,1] \to [0,1]$ by
\[
\tau(x) = F(x) \bmod 1.
\]
Then $\tau$ can be written as a piecewise map with two branches:
\[
\tau(x) =
\begin{cases}
2x + \eta \sin(2\pi x), & x \in [0,a), \\
2x + \eta \sin(2\pi x) - 1, & x \in [a,1].
\end{cases}
\]
Each branch extends to a $C^{1+\varepsilon}$ function on its closure, since $\sin(2\pi x)$ is smooth. Moreover,
\[
|\tau'(x)| = |F'(x)| \ge 2 - 2\pi \eta =: s > 1,
\]
so $\tau$ is uniformly expanding. Since $F'$ is smooth, it is in particular Hölder continuous, hence each branch is $C^{1+\varepsilon}$.

Furthermore, $\tau$ is strictly increasing on each branch. Since
\[
\tau(0) = 0, \quad \tau(a) = 1,
\]
the first branch maps $[0,a]$ onto $[0,1]$. Similarly,
\[
\tau(a) = 0, \quad \tau(1) = 1,
\]
so the second branch maps $[a,1]$ onto $[0,1]$. Thus both branches are full.

Therefore, $\tau \in \mathcal{T}([0,1]; s, \varepsilon)$ and is topologically mixing. By Theorem~4.3, the invariant densities $h_\delta$ of the perturbed operators $P_\delta = Q_\delta P_\tau$ converge in $L^1(I)$ to the invariant density $h$ of $P_\tau$ as $\delta \to 0$.
\end{example}
This example shows that our results apply to genuinely nonlinear expanding systems. In particular, the use of the space $BV_{1,1/p}(I)$ allows one to treat maps with limited regularity beyond the classical bounded variation setting.
\section{Remarks and Further Directions}

In this paper, we established stochastic stability of absolutely continuous invariant measures for piecewise expanding $C^{1+\varepsilon}$ maps in the framework of generalized bounded variation spaces $BV_{1,1/p}(I)$. The key ingredient is a Lasota--Yorke inequality uniform with respect to the perturbation parameter (Thm.~\ref{thm:uniform-LY}), which controls the spectral behavior of the perturbed Frobenius--Perron operators $P_\delta$.

 These results demonstrate that stochastic stability persists under minimal $C^{1+\varepsilon}$
  regularity assumptions ($\varepsilon > 0$). They extend the stochastic stability result
  of G\'ora~\cite{gora1984}---which treated the same map class via Rychlik's $C_\varepsilon$
  seminorm under more restrictive expansion conditions---to the $BV_{1,1/p}(I)$ framework,
  and complement the spectral stability theory of~\cite{keller-liverani}.

Several directions for further research remain open:
\begin{enumerate}
\item Higher dimensions: piecewise expanding maps on $\mathbb{R}^d$ or manifolds, adapting $BV_{1,1/p}(I)$ via Whitney jets.
\item Non-uniform expansion: indifferent fixed points or critical points, combining with indifferent Lasota--Yorke estimates.
\item Refined spaces: $BV_{t,1/p}$ for $t>1$, exploring regularity/stability trade-offs.
\end{enumerate}

These extensions could impact applications in stochastic numerics and data assimilation for low-regularity dynamics.

\end{document}